\documentclass[11pt]{article}

\usepackage[utf8]{inputenc}
\usepackage[a4paper,margin=1.08in]{geometry}
\usepackage{amsmath,amssymb,amsthm,mathtools}
\usepackage{microtype}
\usepackage{enumitem}
\usepackage{hyperref}
\usepackage[nameinlink,noabbrev]{cleveref}

\hypersetup{
  colorlinks=true,
  linkcolor=blue,
  citecolor=blue,
  urlcolor=blue,
  pdftitle={Completeness of the Model Space Does Not Force Regularity for Infinite-Dimensional Lie Groups},
  pdfauthor={Zongjian Han and Fungo Saluta},
  pdfsubject={Infinite-dimensional Lie groups and continuous inverse algebras},
  pdfkeywords={regular Lie group, semiregularity, BCH-Lie group, continuous inverse algebra, complete locally convex space, graded nil algebra}
}

\newtheorem{theorem}{Theorem}[section]
\newtheorem{proposition}[theorem]{Proposition}
\newtheorem{lemma}[theorem]{Lemma}
\newtheorem{corollary}[theorem]{Corollary}
\theoremstyle{remark}
\newtheorem{remark}[theorem]{Remark}

\title{Completeness of the Model Space Does Not Force Regularity for Infinite-Dimensional Lie Groups}
\author{%
Zongjian Han\textsuperscript{1,*,$\dagger$}
\quad
Fungo Saluta\textsuperscript{2,*,$\dagger$}\\[0.65em]
\small\textsuperscript{1}School of Mathematical Sciences, Tongji University,
Shanghai 200092, China\\
\small\textsuperscript{2}Department of Philosophy, Carnegie Mellon University,
Pittsburgh, PA 15213, USA
}
\date{}

\begin{document}
\maketitle

\begingroup
\renewcommand{\thefootnote}{\fnsymbol{footnote}}
\footnotetext[1]{Corresponding authors. E-mail addresses:
\href{mailto:dbln@tongji.edu.cn}{\texttt{dbln@tongji.edu.cn}} (Z. Han);
\href{mailto:fungol@andrew.cmu.edu}{\texttt{fungol@andrew.cmu.edu}} (F. Saluta).}
\endgroup

\begin{abstract}
We give negative solutions to the two basic completeness problems for locally
convex Lie groups: whether a Lie group modelled on a Mackey-complete locally
convex space must be regular, and whether it must at least possess a smooth
exponential function.  Both counterexamples have complete model spaces.  We
first construct a contractible complex analytic BCH--Lie group $H$, modelled
on the complete Silva space $\mathbb C^{(\mathbb N)}$, for which $\exp_H$ is a
global homeomorphism although $H$ is not $C^0$-semiregular.  The failure is
witnessed by smooth controls converging to zero in a fixed finite-dimensional
subspace; their evolutions exist uniquely on $[0,1)$ but have no endpoint at
time $1$.  The obstruction is a multiplicative graded escape in the principal
unit group of a complete continuous inverse algebra.  We then suspend one
such control: the translation action on $C^\infty(S^1,H)$ converts the
time-dependent obstruction into a single element of a semidirect-product Lie
algebra.  The resulting Lie group is modelled on a complete Hausdorff locally
convex space and contains an element which generates no one-parameter
subgroup.  Hence it admits no exponential function.  Thus completeness forces
neither non-autonomous evolution nor autonomous exponentiation.
\end{abstract}

\medskip
\noindent\textbf{2020 Mathematics Subject Classification.}
Primary 22E65; Secondary 46H05, 46A13, 16N40.

\smallskip
\noindent\textbf{Key words and phrases.}
Infinite-dimensional Lie group, Milnor regularity, semiregularity, BCH--Lie
group, continuous inverse algebra, strict LF-space, graded nil algebra.

\section{Introduction and main results}

\label{sec:introduction}

For a finite-dimensional Lie group, every element of the Lie algebra generates
a one-parameter subgroup, and every smooth time-dependent Lie-algebra curve
has a product integral.  The same conclusions hold for Banach--Lie groups by
the ordinary theory of differential equations.  Neither conclusion follows
from the axioms of a Lie group modelled on a general locally convex space:
smooth vector fields on such spaces need not possess integral curves.  Two
integration problems therefore lie at the foundation of infinite-dimensional
Lie theory.  The autonomous problem asks for an exponential function.  The
non-autonomous problem asks for evolution along arbitrary smooth
Lie-algebra curves.

Milnor placed the second problem at the centre of his 1983 Les Houches
lectures.  His definition of a regular Lie group requires the existence of
product integrals for all smooth curves together with smooth dependence of
the endpoint on the curve
\cite[Definition~7.6, p.~1046]{Milnor1984}.  He emphasized that regularity is
a strengthening of the existence and smoothness of the exponential map, and
observed that every Lie group then known was regular
\cite[p.~1046]{Milnor1984}.  The common completeness question emerging from
this framework was later separated by Neeb into two explicit problems:

\begin{enumerate}[label=\textup{(\roman*)},leftmargin=2.2em]
\item must every Lie group modelled on a Mackey-complete locally convex space
      possess a smooth exponential function;
\item must every such Lie group be regular?
\end{enumerate}

These are Problems~II.1 and II.2 of \cite{Neeb2006}; the second is explicitly
attributed there to Milnor.  More than four decades after Milnor's lectures,
the 2026 monograph of Gl\"ockner and Neeb still described a general
regularity or exponential-existence theorem for complete model spaces as a
fundamental open problem and the construction of a non-regular Lie group on a
Mackey-complete model as one of the central open problems of the theory
\cite[Introduction, p.~9; discussion preceding Proposition~6.3.2,
p.~458]{GlocknerNeebBook2026}.

The present paper gives negative answers to both questions, under the stronger
hypothesis that the model space is complete.  The two counterexamples are
linked by a suspension construction, but they exhibit sharply different
behaviour.  The first has a global exponential homeomorphism and fails only at
the level of time-dependent evolution.  The suspension converts this
non-autonomous failure into a single non-integrable Lie-algebra element and
thereby destroys the existence of an exponential function itself.

Let $G$ be a Lie group modelled on a Hausdorff locally convex space, with Lie
algebra $\mathfrak g:=\mathbf L(G)$.  In the standard left convention, an
evolution of $\xi\in C^k([0,1],\mathfrak g)$ is a curve $\Gamma$ satisfying
\[
  \Gamma(0)=e,
  \qquad
  \Gamma'(t)=T_eL_{\Gamma(t)}\xi(t).
\]
The group is $C^k$-semiregular if every such $\xi$ has an evolution, and it is
$C^k$-regular if the endpoint depends smoothly on $\xi$
\cite{Milnor1984,GlocknerSemiregularity,GlocknerNeebBook2026}.  Regularity
implies the existence of a smooth exponential function.  Semiregularity
already implies that every Lie-algebra element generates a one-parameter
subgroup.  We shall also use the equivalent right convention
\begin{equation}
  U(0)=e,
  \qquad
  U'(t)=T_eR_{U(t)}\gamma(t).
  \label{eq:right-evolution}
\end{equation}
Indeed, if $U'=\gamma U$ and $V:=U^{-1}$, then
\[
  V'=V(-\gamma).
\]

Our first main result settles the regularity problem at a level stronger than
non-regularity.

\begin{theorem}[A complete BCH--Lie group without semiregularity]
\label{thm:existence}
There exists a contractible complex analytic BCH--Lie group $H$, modelled on
the complete Silva space $\mathbb C^{(\mathbb N)}$, such that
\[
  \exp_H:\mathbf L(H)\longrightarrow H
\]
is a homeomorphism, while $H$ is not $C^0$-semiregular.  More precisely,
there exist a finite-dimensional subspace $W\subseteq\mathbf L(H)$ and smooth
curves
\[
  \gamma_\lambda\in C^\infty([0,1],W),
  \qquad 0<\lambda\leq1,
\]
with $\gamma_\lambda\to0$ in $C^\infty([0,1],W)$ as $\lambda\to0$, such that
\[
  U'(t)=\gamma_\lambda(t)U(t),
  \qquad U(0)=1,
\]
has no $H$-valued $C^1$ solution on $[0,1]$.  Each equation has a unique
smooth solution on $[0,1)$.
\end{theorem}

Thus the autonomous integration of every Lie-algebra element can be globally
well behaved while arbitrarily small controls in one fixed finite-dimensional
space fail to evolve.  The example is contractible, complex analytic and
BCH; its exponential function is globally bijective with continuous inverse.
The obstruction is therefore independent of covering phenomena, global
topology, failure of BCH coordinates, or the absence of autonomous flows.  It
also has a strong local form: on the fixed space $W$, the domain of smooth
controls admitting evolution contains no neighbourhood of zero and is not
closed, and the endpoint map is discontinuous at zero; see
Corollary~\ref{cor:evolution-domain}.

The autonomous problem has a negative answer as well.

\begin{theorem}[A complete model space without an exponential function]
\label{thm:nonexponentiable}
There exists a smooth Lie group $\widetilde G$ modelled on a complete Hausdorff
locally convex space and an element $g\in\mathbf L(\widetilde G)$ for which
\[
  U'(t)=T_eL_{U(t)}g,
  \qquad U(0)=e,
\]
has no $C^1$ solution on $[0,1]$.  Hence $g$ generates no one-parameter
subgroup and $\widetilde G$ admits no smooth exponential function.  One may
take
\[
  \widetilde G=C^\infty(S^1,H)\rtimes_\alpha\mathbb R,
\]
where $H$ is supplied by Theorem~\ref{thm:existence} and $\alpha$ is the
translation action on the loop variable.
\end{theorem}

The passage from Theorem~\ref{thm:existence} to
Theorem~\ref{thm:nonexponentiable} is the conceptual second step of the
paper.  Choose a bad control which is flat at both endpoints and regard it as
a smooth loop $\widetilde\gamma:S^1\to\mathbf L(H)$.  In the semidirect
product above, the single element $(\widetilde\gamma,1)$ records the entire
time-dependent control: the real coordinate advances the loop variable.  If
$(\widetilde\gamma,1)$ generated a one-parameter subgroup
$U(t)=(F_t,t)$, then evaluating $F_t$ at the fixed point $0\in S^1$ would
produce the forbidden evolution in $H$.  Suspension therefore turns a
non-autonomous obstruction into an autonomous one without imposing any new
incompleteness on the model space.

We now describe the algebraic source of the first counterexample.  Throughout,
algebras are associative complex algebras and need not be unital.  Let
\[
  A=\bigoplus_{n\geq1}A_n,
  \qquad A_iA_j\subseteq A_{i+j},
\]
be a positively graded nil algebra generated by its finite-dimensional
component $A_1$.  Endow its unitization
\begin{equation}
  \mathcal B_A:=\mathbb C1\oplus A
  \label{eq:unitization}
\end{equation}
with the finest Hausdorff locally convex topology, and write
\[
  G_A:=\mathcal B_A^\times,
  \qquad
  H_A:=1+A.
\]
The construction is summarized by the following theorem.

\begin{theorem}[Graded escape]
\label{thm:graded-escape}
Let $A=\bigoplus_{n\geq1}A_n$ be an infinite-dimensional positively graded
complex nil algebra generated by its finite-dimensional component $A_1$.
Then $\mathcal B_A$ is a complete complex continuous inverse algebra.  The
principal unit group $H_A=1+A$ is a contractible complex analytic BCH--Lie
group, modelled on the complete Silva space
$A\cong\mathbb C^{(\mathbb N)}$, and
\[
  \exp_{H_A}:A\longrightarrow H_A
\]
is a homeomorphism.  The full unit group $G_A=\mathcal B_A^\times$ is also a
complex analytic BCH--Lie group.

There exist a finite-dimensional subspace $W\subseteq A=\mathbf L(H_A)$ and
curves
\[
  \gamma_\lambda\in C^\infty([0,1],W),
  \qquad 0<\lambda\leq1,
\]
such that $\gamma_\lambda\to0$ in $C^\infty([0,1],W)$ as $\lambda\to0$,
whereas
\begin{equation}
  U'(t)=\gamma_\lambda(t)U(t),
  \qquad U(0)=1,
  \label{eq:bad-ivp-intro}
\end{equation}
has no $C^1$ solution with values in $H_A$ or in $G_A$.  It has a unique
smooth $H_A$-valued solution on $[0,1)$.  Consequently, $H_A$ and $G_A$ are
not $C^k$-semiregular for any $k\in\mathbb N\cup\{\infty\}$.
\end{theorem}

\begin{corollary}[Continuous inverse algebra consequence]
\label{cor:cia-consequence}
There exists a complete complex continuous inverse algebra whose full unit
group is a complex analytic BCH--Lie group and is not $C^0$-semiregular.  This
algebra does not satisfy the Gl\"ockner--Neeb multiplication-growth
condition.
\end{corollary}

Gl\"ockner and Neeb proved regularity of the unit group of a Mackey-complete
continuous inverse algebra under their multiplication-growth condition and
asked whether that condition is automatic under Mackey completeness
\cite[Remark~E(b)]{GlocknerNeebCIA}.  Corollary~\ref{cor:cia-consequence}
gives a negative answer even for complete algebras.

The mechanism behind Theorem~\ref{thm:graded-escape} is purely
multiplicative.  Every continuous curve on a compact interval has bounded
image in $\mathcal B_A$, and every bounded subset of this finest locally
convex space is contained in a finite-dimensional subspace.  Linear
integration is therefore available.  The controls themselves remain in the
fixed finite-dimensional space $W$.  What leaves every finite-dimensional
stage is the ordered product generated by evolution.

More precisely, one extracts from a finite subset of $A_1$ a sequence
$b_1,b_2,\ldots$ such that
\begin{equation}
  w_n:=b_nb_{n-1}\cdots b_1\neq0
  \qquad(n\geq1).
  \label{eq:word-intro}
\end{equation}
For rapidly decreasing $\varepsilon_n>0$, the forced endpoints at times
$t_n=1-2^{-n}$ are
\[
  P_n(\lambda)
  =(1+\lambda\varepsilon_nb_n)\cdots
   (1+\lambda\varepsilon_1b_1).
\]
Their degree-$n$ components satisfy
\begin{equation}
  \pi_n(P_n(\lambda))
  =\lambda^n\varepsilon_1\cdots\varepsilon_nw_n\neq0.
  \label{eq:top-degree-intro}
\end{equation}
If $(P_n(\lambda))$ converged, then it would be bounded and hence lie in
a finite-dimensional subspace.  Such a subspace involves only finitely many
homogeneous degrees, contradicting Equation~\eqref{eq:top-degree-intro}.
The pulse amplitudes decay fast enough for the control to extend smoothly to
the accumulation time, with all derivatives vanishing there.  Each finite
stage is therefore a valid smooth evolution, while finite-time accumulation
creates a new nonzero homogeneous degree at every step.  This is the graded
escape.

Golod--Shafarevich algebras provide the required infinite-dimensional graded
nil algebras and may be chosen two-generated.  Within the class considered
here, Section~\ref{sec:consequences} proves a sharp dichotomy: finite
dimensionality, $C^0$-semiregularity, Milnor regularity of both $H_A$ and
$G_A$, and the Gl\"ockner--Neeb multiplication-growth condition are
equivalent.

The paper is organized as follows.  Section~\ref{sec:algebra} constructs the
nonvanishing infinite word.  Section~\ref{sec:cia} establishes the complete
locally convex algebra and Lie-group structures.  Section~\ref{sec:control}
builds the smooth finite-dimensional pulse family and proves the absence of a
global evolution.  Section~\ref{sec:consequences} derives the instability and
classification results.  The final section carries out the suspension and
proves Theorem~\ref{thm:nonexponentiable}.

\section{The graded nil algebra input}
\label{sec:algebra}

The homogeneous component $A_n$ is an image of $A_1^{\otimes n}$; hence
\begin{equation}
  \dim A_n\leq(\dim A_1)^n<\infty.
  \label{eq:automatic-finite-dimensionality}
\end{equation}
In particular, $A$ and $\mathcal B_A$ are countable-dimensional.

We recall this form of K\"onig's lemma.

\begin{lemma}[K\"onig]
\label{Konig}
  
Let $\mathcal{T}$ be an infinite rooted tree such that every vertex has finite degree.
Then $\mathcal{T}$ contains an infinite branch.

\end{lemma}

\begin{proof}

Starting at the root, choose recursively a child whose descendant subtree is
infinite.  Such a child exists at every stage because the current vertex has
only finitely many children, whereas its descendant subtree is infinite.  The
chosen vertices form an infinite branch.

\end{proof}

\begin{lemma}[Nonvanishing branch]
\label{lem:branch}
Let $A=\bigoplus_{n\geq1}A_n$ be a infinite-dimensional positively graded
algebra generated by a finite set $S\subseteq A_1$.  Then there is a sequence
$b_1,b_2,\ldots\in S$ such that
\[
  b_n\cdots b_1\neq0
  \qquad(n\geq1).
\]
\end{lemma}

\begin{proof}
If $A_N=0$ for some $N$, then every product of $m\geq N$ elements of $S$
vanishes, because its rightmost block of length $N$ lies in $A_N$.  Since $S$
generates $A$, this gives $A_m=0$ for all $m\geq N$.  Together with
\eqref{eq:automatic-finite-dimensionality}, this would make $A$
finite-dimensional.  Thus $A_n\neq0$ for every $n$.

For $n\geq1$, let
\[
  \mathcal T_n
  :=\{(a_n,\ldots,a_1)\in S^n:a_n\cdots a_1\neq0\}.
\]
Each $\mathcal T_n$ is nonempty because $A_n$ is spanned by products of $n$
elements of $S$.  If $(a_n,\ldots,a_1)\in\mathcal T_n$, then
$a_{n-1}\cdots a_1\neq0$; otherwise left multiplication by $a_n$ would give
$a_n\cdots a_1=0$.  Deleting the leftmost letter defines a finitely branching
rooted tree with nonempty levels.  K\"onig's lemma yields an infinite branch.
\end{proof}

The required algebra is supplied by the classical Golod--Shafarevich
construction.

\begin{theorem}[Golod--Shafarevich]
\label{thm:golod}
There exists an infinite-dimensional positively graded complex nil algebra
$A$ generated by its degree-one component, with
\[
  \dim_{\mathbb C}A_1=2.
\]
\end{theorem}

\begin{proof}[Source]
Take $d=2$ in the homogeneous Golod--Shafarevich construction of
\cite[Theorem~1.1 and Section~4]{RegevRegev2009}.  Thus
$T_{\geq1}=\mathbb C\langle x_1,x_2\rangle_{\geq1}$ has a homogeneous
two-sided ideal $I$, generated in degrees at least two, such that
\[
  A:=T_{\geq1}/I
\]
is nil and infinite-dimensional.  The quotient inherits the positive grading,
and $I\cap T_1=0$ gives $A_1\cong\mathbb C^2$.  This is the graded form of
Golod's original construction \cite{Golod1964}.
\end{proof}

\begin{corollary}
\label{cor:existence}
There exists an algebra satisfying the hypotheses of
Theorem~\ref{thm:graded-escape}.
\end{corollary}

\begin{proof}
Apply Theorem~\ref{thm:golod}.
\end{proof}

\begin{remark}[Minimal noncommutative input]
\label{rem:noncommutative-input}
The algebra in Theorem~\ref{thm:golod} can be chosen two-generated, and this
is minimal.  A nil algebra generated by one element $g$ is spanned by
$x,x^2,\ldots,x^{N-1}$ once $x^N=0$.  More generally, if a commutative nil
algebra is generated by $x_1,\ldots,x_d$ and $x_j^{N_j}=0$, then it is spanned
by the finitely many monomials $x_1^{m_1}\cdots x_d^{m_d}$ with
$0\leq m_j<N_j$.  Thus the infinite homogeneous escape used below is a
genuinely noncommutative phenomenon.
\end{remark}

\begin{remark}
It is worth noting that it has already been shown in 
\cite[Theorem~6.3.3, p.~458]{GlocknerNeebBook2026}
that all Mackey complete connected abelian Lie groups are regular,
as they are precisely the quotients $\mathfrak{a}/\Gamma$,
where $(\mathfrak{a},+)$ is a Mackey complete vector space 
and $\Gamma \subseteq \mathfrak{a}$ a discrete subgroup.

\end{remark}

\section{The principal unit group on the complete model space}
\label{sec:cia}

Let $A=\bigoplus_{n\geq1}A_n$ be a positively graded complex nil algebra
generated by its finite-dimensional component $A_1$, and let
$\tau_{\mathrm f}$ denote the finest Hausdorff locally convex topology on
$\mathcal B_A=\mathbb C1\oplus A$.

\begin{lemma}[The induced topology on the graded ideal]
\label{lem:induced-topology}
The subspace topology on $A\subseteq\mathcal B_A$ is the finest Hausdorff
locally convex topology on $A$.  Moreover,
\[
  \mathcal B_A=\mathbb C1\oplus A
\]
is a topological direct sum, and the inclusion $A\to\mathcal B_A$ is a
complemented linear embedding.
\end{lemma}

\begin{proof}
Let $p$ be a seminorm on $A$.  The formula
\[
  \widetilde p(\lambda1+a):=p(a)
\]
defines a seminorm on $\mathcal B_A$, hence a continuous seminorm for
$\tau_{\mathrm f}$.  Its restriction to $A$ is $p$.  Thus every seminorm on
$A$ is continuous for the induced topology.  Since the induced topology is
locally convex and the finest locally convex topology is maximal among such
topologies, the two coincide.  The scalar projection
$\pi_0(\lambda1+a)=\lambda$ and the projection
$\pi_A(\lambda1+a)=a$ are continuous linear maps, which gives the asserted
topological direct sum.
\end{proof}

\begin{lemma}[The model space]
\label{lem:model-space}
Let $E$ be a countable-dimensional complex vector space with its
finest Hausdorff locally convex topology.  For a Hamel basis
$(e_j)_{j\geq1}$, set
\[
  E_n:=\operatorname{span}_{\mathbb C}\{e_1,\ldots,e_n\}.
\]
Then
\[
  E=\varinjlim E_n
\]
as a locally convex space and as a topological direct limit.  It is a strict
LF-space and a Silva space; in particular, it is complete.
\end{lemma}

\begin{proof}
Every linear map from a finite-dimensional locally convex space to a locally
convex space is continuous.  Hence every locally convex vector topology on
$E$ makes the inclusions $E_n\to E$ continuous.  The locally convex direct
limit topology is therefore the finest Hausdorff locally convex topology on
$E$.  The inclusions $E_n\to E_{n+1}$ are closed topological embeddings, so
$E$ is a strict LF-space.  Strict LF-spaces are complete
\cite[Chapter~II, \S6, pp.~59--61]{SchaeferWolff}.

Each inclusion $E_n\to E_{n+1}$ is compact, so $E$ is also a Silva space.
Its locally convex topology agrees with the topological direct-limit topology
\cite[Proposition~B.13.13(a)]{GlocknerNeebBook2026}; the special case of the
finest countable-dimensional topology is also recorded in
\cite[Appendix~D]{NeebRusso2024}.  Any two countably
infinite-dimensional complex vector spaces with their finest locally convex
topologies are topologically isomorphic, because a linear bijection and its
inverse are continuous.  Thus $E\cong\mathbb{C}^{(\mathbb{N})}$.
\end{proof}

\begin{lemma}[Open subsets of a final union]
\label{lem:open-final}
Suppose that $E=\bigcup_nE_n$ has the final topology with respect to the
inclusions $E_n\to E$, and let $U\subseteq E$ be open.  Then $U$ has the final
topology with respect to the inclusions $U\cap E_n\to U$.
\end{lemma}

\begin{proof}
Let $O\subseteq U$ and assume that $O\cap E_n$ is open in $U\cap E_n$ for
every $n$.  Since $U\cap E_n$ is open in $E_n$, the set $O\cap E_n$ is open
in $E_n$.  The final topology on $E$ makes $O$ open in $E$, hence open in
$U$.  The converse follows from continuity of the inclusions.
\end{proof}

\begin{proposition}
\label{prop:topological-algebra}
The space $(\mathcal B_A,\tau_{\mathrm f})$ is complete, and multiplication
\[
  \mathcal B_A\times\mathcal B_A\to\mathcal B_A,
  \qquad (x,y)\mapsto xy,
\]
is jointly continuous.  Every bounded subset of $\mathcal B_A$ lies in a
finite-dimensional subspace.
\end{proposition}

\begin{proof}
If $\mathcal B_A$ is finite-dimensional, completeness, joint continuity of
multiplication, and the assertion concerning bounded sets are immediate.
Assume henceforth that $\mathcal B_A$ is infinite-dimensional.  Completeness
follows from Lemma~\ref{lem:model-space}.

Choose a complex Hamel basis $(e_j)_{j\geq1}$.  Let $p$ be a seminorm on
$\mathcal B_A$ and put $a_{ij}:=p(e_ie_j)$.  Choose positive numbers
$\alpha_i,\beta_j$ such that $a_{ij}\leq\alpha_i\beta_j$ for all $i,j$; for
example,
\[
  \alpha_i:=1+\max_{1\leq j\leq i}2^j a_{ij},
  \qquad
  \beta_1:=\frac12,
  \qquad
  \beta_j:=\max\left\{2^{-j},
  \max_{1\leq i<j}\frac{a_{ij}}{\alpha_i}\right\}
  \quad(j\geq2).
\]
If $i<j$, the second term in the definition of $\beta_j$ gives
$a_{ij}\leq\alpha_i\beta_j$.  If $i\geq j$, then
$\alpha_i\geq2^ja_{ij}$ and $\beta_j\geq2^{-j}$, including $j=1$.
Thus $a_{ij}\leq\alpha_i\beta_j$ in all cases.  Define
\[
  q\left(\sum_i x_ie_i\right):=\sum_i\alpha_i|x_i|,
  \qquad
  r\left(\sum_j y_je_j\right):=\sum_j\beta_j|y_j|.
\]
All seminorms are continuous for $\tau_{\mathrm f}$, and
\[
  p(xy)\leq q(x)r(y).
\]
Thus multiplication is jointly continuous.

Suppose that a bounded set $D\subseteq\mathcal B_A$ is not contained in a
finite-dimensional subspace.  Choose linearly independent elements
$d_1,d_2,\ldots\in D$, extend them to a Hamel basis, and define a seminorm
$p_D$ by $p_D(d_j)=j$ and weight $1$ on the remaining basis vectors.  The
seminorm is continuous for $\tau_{\mathrm f}$, whereas $p_D(D)$ is
unbounded.  This contradiction proves the final assertion.
\end{proof}

\begin{lemma}[Uniform nilpotence]
\label{lem:uniform-nilpotence}
For every finite-dimensional subspace $F\subseteq A$, there is an integer
$N_F$ such that
\[
  a^{N_F}=0
  \qquad(a\in F).
\]
\end{lemma}

\begin{proof}
Set $Z_m:=\{a\in F:a^m=0\}$.  Each $Z_m$ is closed and
$F=\bigcup_{m\geq1}Z_m$.  The Baire theorem gives an $N_F$ for which
$Z_{N_F}$ has nonempty interior.  Let $P_{N_F}:F\to A$ be the polynomial map
$P_{N_F}(a)=a^{N_F}$.  If $P_{N_F}(a_0)\neq0$ for some $a_0\in F$, choose an
algebraic linear functional $\ell\in A^*$ with
$\ell(P_{N_F}(a_0))\neq0$.  The scalar polynomial $\ell\circ P_{N_F}$
vanishes on the nonempty open set $\operatorname{int}(Z_{N_F})$, hence is
identically zero, a contradiction.
\end{proof}

\begin{theorem}
\label{thm:cia-and-principal}
The algebra $\mathcal B_A$ is a complete complex continuous inverse algebra.
Its unit group $G_A=\mathcal B_A^\times$ is a complex analytic BCH--Lie group.
The principal unit group $H_A=1+A$ is a closed normal split embedded complex
analytic BCH--Lie subgroup modelled on the complete locally convex space $A$.
If $A$ is infinite-dimensional, then $A\cong\mathbb{C}^{(\mathbb{N})}$ is a strict
LF-space and a Silva space.  In fact,
\[
  G_A\cong\mathbb C^\times\times H_A
\]
as complex analytic Lie groups.  The group $H_A$ is contractible.
\end{theorem}

\begin{proof}
An element $u=\lambda1+a$ is invertible precisely when $\lambda\neq0$.  If
$a^N=0$, then
\begin{equation}
  (\lambda1+a)^{-1}
  =\lambda^{-1}\sum_{k=0}^{N-1}(-\lambda^{-1}a)^k.
  \label{eq:inverse-formula}
\end{equation}
Hence
\[
  G_A=\{\lambda1+a:\lambda\in\mathbb C^\times,\ a\in A\}
\]
is open, since the scalar projection
$\pi_0(\lambda1+a)=\lambda$ is continuous.

If $A$ is finite-dimensional, Lemma~\ref{lem:uniform-nilpotence} and
Formula~\eqref{eq:inverse-formula} give a single finite rational-polynomial
formula for inversion on $G_A$, hence continuity.  Suppose that $A$ is
infinite-dimensional.  Use a basis exhaustion $E_n$ as in
Lemma~\ref{lem:model-space}.  The finite-dimensional space
$F_n:=\pi_A(E_n)$ has a uniform nilpotence index by
Lemma~\ref{lem:uniform-nilpotence}.  Formula~\eqref{eq:inverse-formula}
therefore gives a finite rational-polynomial expression for inversion on
$E_n\cap G_A$.  Lemma~\ref{lem:open-final} and the final topology of
Lemma~\ref{lem:model-space} imply continuity of inversion on $G_A$.  Thus
$\mathcal B_A$ is a complete complex continuous inverse algebra.

By \cite[Theorem~5.6]{GlocknerCIA2002}, inversion is analytic and $G_A$
is a complex analytic Lie group.  Since $\mathcal B_A$ is complete, hence
Mackey-complete, the same theorem gives the BCH property.

By Lemma~\ref{lem:induced-topology}, the affine chart
$A\to H_A$, $a\mapsto1+a$, uses the finest locally convex topology on $A$.
In this chart multiplication is
\[
  a*b=a+b+ab,
\]
and inversion is the restriction of the analytic inversion on $G_A$.
Accordingly, $H_A$ is a complex analytic Lie group modelled on $A$.  If $A$
is infinite-dimensional, Lemma~\ref{lem:model-space} identifies this model
space with the complete strict LF-space $\mathbb{C}^{(\mathbb{N})}$.

The map
\[
  \Theta:\mathbb C^\times\times H_A\to G_A,
  \qquad (\lambda,1+a)\mapsto\lambda(1+a),
\]
is a group isomorphism with inverse
\[
  \Theta^{-1}(\lambda1+a)
  =\bigl(\lambda,1+\lambda^{-1}a\bigr).
\]
Both maps are analytic in the global affine charts, because scalar
multiplication is bilinear and $\lambda\mapsto\lambda^{-1}$ is analytic on
$\mathbb C^\times$.  Thus $H_A$ is a closed normal split embedded analytic direct factor of $G_A$;
in the global affine charts, its inclusion is represented by the complemented
linear embedding $A\hookrightarrow\mathbb C1\oplus A$.

Let $\exp_{G_A}$ denote the exponential map of the BCH--Lie group $G_A$.
Since
\[
  \pi_0:G_A\to\mathbb C^\times
\]
is an analytic homomorphism, naturality of exponential maps gives
\[
  \pi_0(\exp_{G_A}x)
  =\exp_{\mathbb C^\times}(\pi_0x)
  =e^{\pi_0(x)}.
\]
Choose a BCH neighborhood $U\subseteq\mathcal B_A$ on which
$\exp_{G_A}:U\to V$ is an analytic diffeomorphism, and shrink $U$ so that
$|\pi_0(x)|<\pi$ for $x\in U$.  Then
\[
  \exp_{G_A}(U\cap A)=V\cap H_A.
\]
The inclusion from left to right follows from naturality.  Conversely, if
$\exp_{G_A}(x)\in H_A$, then $e^{\pi_0(x)}=1$; the bound
$|\pi_0(x)|<\pi$ implies $\pi_0(x)=0$.  The restricted chart proves that
$H_A$ is a BCH--Lie group.

Finally,
\[
  [0,1]\times H_A\to H_A,
  \qquad (s,1+a)\mapsto1+sa,
\]
is a smooth contraction to the identity.  Thus $H_A$ is contractible.
\end{proof}

\begin{corollary}[Global exponential homeomorphism]
\label{cor:global-exponential}
The exponential map
\[
  \exp_{H_A}:A\longrightarrow H_A
\]
is a homeomorphism.
\end{corollary}

\begin{proof}
For $a\in A$, choose $N$ with $a^N=0$ and set
\[
  E(a):=1+\sum_{j=1}^{N-1}\frac{a^j}{j!}.
\]
The curve $t\mapsto E(ta)$ is a smooth one-parameter subgroup of $H_A$ with
derivative $a$ at zero, so $E(a)=\exp_{H_A}(a)$.  Conversely, for
$1+x\in H_A$, choose $M$ with $x^M=0$ and define
\[
  L(1+x):=\sum_{j=1}^{M-1}\frac{(-1)^{j+1}}{j}x^j.
\]
All terms lie in the commutative finite-dimensional algebra generated by the
single element $g$.  The formal identities for exponential and logarithm in
a sufficiently high truncated polynomial ring give
\[
  E(L(1+x))=1+x,
  \qquad
  L(E(a))=a.
\]
Thus $\exp_{H_A}$ is bijective with inverse $L$.

It remains to prove continuity of $L$.  If $A$ is finite-dimensional,
Lemma~\ref{lem:uniform-nilpotence}, applied to $F=A$, gives an integer $N$
such that $x^N=0$ for every $x\in A$.  Hence
\[
  L(1+x)=\sum_{j=1}^{N-1}\frac{(-1)^{j+1}}{j}x^j
\]
is one fixed polynomial on $H_A$ and is therefore continuous.

Assume now that $A$ is infinite-dimensional.  Choose a finite-dimensional
basis exhaustion $A=\bigcup_nE_n$ as in Lemma~\ref{lem:model-space}.
Lemma~\ref{lem:uniform-nilpotence} gives an integer $N_n$ such that
$x^{N_n}=0$ for every $x\in E_n$.  Hence $L(1+x)$ is a fixed polynomial in
$g$ on $1+E_n$, and is continuous there.  The affine chart $A\to H_A$,
$x\mapsto1+x$, and the topological final property in
Lemma~\ref{lem:model-space} imply that $L:H_A\to A$ is continuous.  The
exponential map is smooth, hence continuous, so it is a homeomorphism.
\end{proof}

\begin{remark}
\label{rem:finite-range}
The image of every continuous curve $c:[0,1]\to\mathcal B_A$ is compact,
hence bounded, and Proposition~\ref{prop:topological-algebra} places it in a
finite-dimensional linear subspace.  Linear integration therefore causes no
obstruction in the example below.  The unbounded homogeneous depth is created
by repeated multiplication.
\end{remark}

\begin{remark}[Vector-space and Lie-group direct limits]
The finite-dimensional exhaustion used in Lemma~\ref{lem:model-space} consists
of vector subspaces.  It need not consist of subalgebras or Lie subgroups.
Accordingly, regularity theorems for direct limits of finite-dimensional Lie
groups, such as \cite{GlocknerDirectLimits2021}, do not apply to this
construction.
\end{remark}

\section{A smooth finite-dimensional control without evolution}
\label{sec:control}

Assume that $A$ is infinite-dimensional.  Fix a basis
$S=\{b^{(1)},\ldots,b^{(d)}\}$ of $A_1$ and a sequence
$(b_n)_{n\geq1}\subseteq S$ supplied by Lemma~\ref{lem:branch}.  Put
\[
  w_n:=b_n\cdots b_1\neq0.
\]
For each $b\in S$, define
\[
  \nu_b:=\min\{m\geq2:b^m=0\};
\]
this is well defined because every $b\in S$ is nonzero and nilpotent.  Set
\begin{equation}
  W:=\operatorname{span}_{\mathbb C}
  \{b^j:b\in S,\ 1\leq j<\nu_b\}.
  \label{eq:control-space}
\end{equation}
Then $W$ is finite-dimensional.

Choose $\chi\in C^\infty([0,1],[0,1])$ which is $0$ near $0$ and $1$ near
$1$.  For $n\geq0$, set
\[
  t_n:=1-2^{-n}.
\]
For $n\geq1$, define
\[
  I_n:=[t_{n-1},t_n],
  \qquad
  \delta_n:=t_n-t_{n-1}=2^{-n},
  \qquad
  \varepsilon_n:=2^{-n^3}.
\]
For $0<\lambda\leq1$, define
\[
  h_{n,\lambda}(s):=1+\lambda\varepsilon_n\chi(s)b_n
  \qquad(0\leq s\leq1)
\]
and
\[
  \eta_{n,\lambda}(s):=
  h_{n,\lambda}'(s)h_{n,\lambda}(s)^{-1}.
\]
With $\nu:=\nu_{b_n}$, one has
\begin{equation}
  \eta_{n,\lambda}(s)
  =\lambda\varepsilon_n\chi'(s)
   \sum_{j=0}^{\nu-2}
   (-\lambda\varepsilon_n\chi(s))^j b_n^{j+1}
  \in W.
  \label{eq:eta-expansion}
\end{equation}
Set
\begin{equation}
  \gamma_\lambda(t):=
  \delta_n^{-1}\eta_{n,\lambda}
  \left(\frac{t-t_{n-1}}{\delta_n}\right)
  \quad(t\in I_n),
  \qquad
  \gamma_\lambda(1):=0.
  \label{eq:gamma-definition}
\end{equation}

\begin{lemma}[Smooth pulse family]
\label{lem:smooth-pulses}
The curves $\gamma_\lambda$ belong to $C^\infty([0,1],W)$ and
\[
  \gamma_\lambda\longrightarrow0
  \quad\text{in }C^\infty([0,1],W)
  \quad(\lambda\downarrow0).
\]
\end{lemma}

\begin{proof}
Fix a norm on $W$.  Since $S$ is finite, differentiation of
\eqref{eq:eta-expansion} gives, for every $r\geq0$, a constant $C_r$ such
that
\[
  \sup_{s\in[0,1]}\|\eta_{n,\lambda}^{(r)}(s)\|
  \leq C_r\lambda\varepsilon_n.
\]
After the affine reparametrization in \eqref{eq:gamma-definition},
\begin{equation}
  \sup_{t\in I_n}\|\gamma_\lambda^{(r)}(t)\|
  \leq C_r\lambda\varepsilon_n\delta_n^{-r-1}
  =C_r\lambda 2^{-n^3+n(r+1)}.
  \label{eq:pulse-derivative-estimate}
\end{equation}
The right-hand side tends to zero as $n\to\infty$.  Since $\chi$ is constant
near both endpoints, all derivatives of the pieces vanish there.  The pieces
therefore fit smoothly at every $t_n$, and
\eqref{eq:pulse-derivative-estimate} gives a smooth extension at $t=1$ with
all derivatives zero.  Taking the supremum over $n$ yields a bound
$C_r'\lambda$, proving convergence in the $C^\infty$ topology.
\end{proof}

Define
\[
  P_0(\lambda):=1,
  \qquad
  P_n(\lambda):=(1+\lambda\varepsilon_n b_n)P_{n-1}(\lambda).
\]
For $m\geq1$, let
\[
  \pi_m:\mathcal B_A=\mathbb C1\oplus\bigoplus_{j\geq1}A_j
  \longrightarrow A_m
\]
denote the canonical homogeneous projection; thus $\pi_m(1)=0$.

\begin{proposition}[Unique evolution before the accumulation time]
\label{prop:local-evolution}
The formula
\begin{equation}
  U_\lambda(t):=
  h_{n,\lambda}\left(\frac{t-t_{n-1}}{\delta_n}\right)
  P_{n-1}(\lambda),
  \qquad t\in I_n,
  \label{eq:local-evolution}
\end{equation}
defines the unique $H_A$-valued $C^1$ solution on $[0,1)$ of
$U_\lambda'=\gamma_\lambda U_\lambda$ with $U_\lambda(0)=1$.  This
solution is smooth and satisfies
\[
  U_\lambda(t_n)=P_n(\lambda).
\]
\end{proposition}

\begin{proof}
On $I_n$, the identity
$h_{n,\lambda}'=\eta_{n,\lambda}h_{n,\lambda}$ and the time
reparametrization give $U_\lambda'=\gamma_\lambda U_\lambda$.  The endpoint
relations $h_{n,\lambda}(0)=1$ and
$h_{n,\lambda}(1)=1+\lambda\varepsilon_n b_n$ make the pieces agree.  Since
$h_{n,\lambda}$ is constant near both endpoints, the resulting path is smooth
on $[0,1)$.

If $V:[0,1)\to H_A$ is another $C^1$ solution with the same initial value,
then
\[
  (V^{-1}U_\lambda)'=0.
\]
Continuous linear functionals separate points, so $V^{-1}U_\lambda$ is
constant.  Its value at $0$ is $1$, and hence $V=U_\lambda$.
\end{proof}

\begin{theorem}[Absence of a global $C^1$ evolution]
\label{thm:no-evolution}
For every $0<\lambda\leq1$, the initial-value problem
\[
  U'(t)=\gamma_\lambda(t)U(t),
  \qquad U(0)=1,
\]
has no $C^1$ solution with values in $G_A$, and hence none with values in
$H_A$.
\end{theorem}

\begin{proof}
Suppose that $U:[0,1]\to G_A$ is such a solution.  The scalar projection
satisfies
\[
  (\pi_0\circ U)'(t)=\pi_0(\gamma_\lambda(t)U(t))=0,
  \qquad \pi_0(U(0))=1,
\]
so $U(t)\in H_A$ for all $t$.  The restriction of $U$ to $[0,1)$ and the
path $U_\lambda$ from Proposition~\ref{prop:local-evolution} solve the same
initial-value problem.  The uniqueness assertion in that proposition gives
$U=U_\lambda$ on $[0,1)$.  In particular,
\begin{equation}
  U(t_n)=P_n(\lambda)
  =(1+\lambda\varepsilon_n b_n)\cdots
   (1+\lambda\varepsilon_1b_1).
  \label{eq:forced-endpoints}
\end{equation}

Each factor in \eqref{eq:forced-endpoints} has homogeneous degrees $0$ and
$1$.  A term of degree $n$ must select the nonconstant term from every one of
the $n$ factors.  Consequently,
\begin{equation}
  \pi_n(P_n(\lambda)-1)
  =\lambda^n\varepsilon_1\cdots\varepsilon_nw_n\neq0.
  \label{eq:new-degree}
\end{equation}
There are no further degree-$n$ terms.

If $(P_n(\lambda))$ converged in $\mathcal B_A$, then it would be bounded.
Proposition~\ref{prop:topological-algebra} would place it in a
finite-dimensional subspace $F$.  A finite basis of $F$ involves only
finitely many homogeneous degrees, so there is an $N$ such that
\[
  F\subseteq\mathbb C1\oplus\bigoplus_{j=1}^N A_j.
\]
This contradicts \eqref{eq:new-degree} for $n>N$.  Thus
$(P_n(\lambda))$ has no limit.  Since $t_n\to1$, continuity of $U$ would give
$P_n(\lambda)=U(t_n)\to U(1)$, a contradiction.
\end{proof}

\begin{corollary}
\label{cor:all-k}
For every $k\in\mathbb N\cup\{\infty\}$, the groups $H_A$ and $G_A$ are
not $C^k$-semiregular.  Every neighbourhood of the zero curve in
$C^k([0,1],W)$ contains a smooth curve admitting neither an $H_A$-valued nor
a $G_A$-valued right evolution.
\end{corollary}

\begin{proof}
The curves $\gamma_\lambda$ are smooth and converge to zero in the
$C^\infty$ topology, hence also in every $C^k$ topology.  A
$C^{k+1}$ evolution would be a $C^1$ solution, which
Theorem~\ref{thm:no-evolution} excludes.  Inversion converts the right
evolution statement to the standard left convention.
\end{proof}

\begin{corollary}[Instability of the evolution domain]
\label{cor:evolution-domain}
Let
\[
  \mathcal D_W:=\{\gamma\in C^\infty([0,1],W):
  U'=\gamma U,\ U(0)=1,
  \text{ has an $H_A$-valued $C^1$ solution}\}.
\]
Then $0\in\mathcal D_W$, the zero curve is not an interior point of
$\mathcal D_W$, and $\mathcal D_W$ is not closed in
$C^\infty([0,1],W)$.  The endpoint map, with $\mathcal D_W$ carrying the
subspace topology,
\[
  \operatorname{evol}_W:\mathcal D_W\to H_A,
  \qquad \gamma\mapsto U_\gamma(1),
\]
is well defined and is not continuous at $0$.
\end{corollary}

\begin{proof}
The uniqueness calculation in the proof of
Proposition~\ref{prop:local-evolution} applies to any two $H_A$-valued
$C^1$ solutions of the same right-evolution equation; hence
$\operatorname{evol}_W$ is well defined.  Moreover, a $C^1$ solution for a
smooth control is automatically smooth.  Indeed, in the affine chart
$H_A\cong A$, write $U=1+u$.  The equation becomes
\[
  u'=\gamma+\gamma u.
\]
Since multiplication is continuous bilinear, its right-hand side is $C^m$
whenever $u$ is $C^m$; induction gives $u\in C^m$ for every $m$.
Thus $\mathcal D_W$ is precisely the usual smooth right-evolution domain
restricted to $C^\infty([0,1],W)$.

The inclusion $0\in\mathcal D_W$ is immediate, and
Corollary~\ref{cor:all-k} shows that $0$ is not an interior point.  Fix
$\lambda>0$.  For $N\geq1$, define
\[
  \gamma_{\lambda,N}(t):=
  \begin{cases}
    \gamma_\lambda(t),&0\leq t\leq t_N,\\
    0,&t_N\leq t\leq1.
  \end{cases}
\]
The pulse construction is identically zero near every joining point, so
$\gamma_{\lambda,N}$ is smooth.  It admits the global evolution
\[
  U_{\lambda,N}(t):=
  \begin{cases}
    U_\lambda(t),&0\leq t\leq t_N,\\
    P_N(\lambda),&t_N\leq t\leq1.
  \end{cases}
\]
For every derivative order $r$, estimate
\eqref{eq:pulse-derivative-estimate} shows that the $C^r$-norm of the omitted
tail tends to zero.  Hence
$\gamma_{\lambda,N}\to\gamma_\lambda$ in
$C^\infty([0,1],W)$.  The limit does not belong to $\mathcal D_W$ by
Theorem~\ref{thm:no-evolution}; thus $\mathcal D_W$ is not closed.

For endpoint continuity, take $\lambda_N:=2^{-N}$ and
$\xi_N:=\gamma_{\lambda_N,N}$.  The uniform bounds in
Lemma~\ref{lem:smooth-pulses} imply $\xi_N\to0$ in
$C^\infty([0,1],W)$, while
\[
  \operatorname{evol}_W(\xi_N)=P_N(\lambda_N).
\]
The set $\{P_N(\lambda_N):N\geq1\}$ is not bounded: if it were bounded,
Proposition~\ref{prop:topological-algebra} would place it in a fixed
finite-dimensional subspace, whereas
\[
  \pi_N(P_N(\lambda_N)-1)
  =\lambda_N^N\varepsilon_1\cdots\varepsilon_Nw_N\neq0
\]
introduces a new homogeneous degree for every $N$.  Therefore the endpoints
do not converge to the identity, and $\operatorname{evol}_W$ is not
continuous at $0$.
\end{proof}

\begin{remark}
\label{rem:accumulation}
Proposition~\ref{prop:local-evolution} gives a unique smooth evolution on
every compact subinterval of $[0,1)$.  The global failure is produced by the
finite-time accumulation of these valid local evolutions: their endpoints
$P_n(\lambda)$ acquire one new homogeneous degree at each step and have no
limit as $t_n\uparrow1$.
\end{remark}

\begin{proof}[Proof of Theorem~\ref{thm:graded-escape}]
Theorem~\ref{thm:cia-and-principal} and
Corollary~\ref{cor:global-exponential} provide the complete continuous
inverse algebra, the complete Silva model, the analytic BCH--Lie groups
$H_A$ and $G_A$, the contractibility of $H_A$, and the global exponential
homeomorphism on $H_A$.

Lemma~\ref{lem:smooth-pulses} and
Proposition~\ref{prop:local-evolution} provide the smooth finite-dimensional
controls and their unique smooth evolutions on $[0,1)$.  Theorem~\ref{thm:no-evolution}
and Corollary~\ref{cor:all-k} give the failure of global evolution and all
semiregularity assertions.
\end{proof}

\section{Growth and classification}
\label{sec:consequences}

\begin{theorem}[Regularity classification]
\label{thm:classification}
Let $A=\bigoplus_{n\geq1}A_n$ be a positively graded complex nil algebra
generated by its finite-dimensional component $A_1$.  For
$\mathcal B_A=\mathbb C1\oplus A$ with the finest Hausdorff locally convex
topology, the following conditions are equivalent:
\begin{enumerate}[label=\textup{(\roman*)}]
\item $A$ is finite-dimensional;
\item $H_A$ is $C^0$-semiregular;
\item $G_A$ is $C^0$-semiregular;
\item $H_A$ is regular in the sense of Milnor;
\item $G_A$ is regular in the sense of Milnor;
\item $\mathcal B_A$ satisfies the Gl\"ockner--Neeb multiplication-growth
condition.
\end{enumerate}
If $A$ is infinite-dimensional, both $H_A$ and $G_A$ fail
$C^k$-semiregularity for every $k\in\mathbb N\cup\{\infty\}$.
\end{theorem}

For $n\geq1$, let
\[
  \mu_n(x_1,\ldots,x_n):=x_1\cdots x_n.
\]
For continuous seminorms $p,q$ on $\mathcal B_A$, write
\[
  \|\mu_n\|_{p,q}
  :=\sup\{p(x_1\cdots x_n):q(x_j)\leq1\text{ for all }j\}.
\]
The Gl\"ockner--Neeb condition requires that, for every continuous seminorm
$p$, there exist a continuous seminorm $q$ and $r>0$ such that
\begin{equation}
  \sum_{n=1}^\infty r^n\|\mu_n\|_{p,q}<\infty
  \label{eq:GN-condition}
\end{equation}
\cite[p.~96]{GlocknerNeebCIA}.

\begin{proposition}
\label{prop:growth}
If $A$ is infinite-dimensional, then $\mathcal B_A$ does not satisfy
\eqref{eq:GN-condition}.  If $A$ is finite-dimensional, then it does.
\end{proposition}

\begin{proof}
Assume first that $A$ is infinite-dimensional and use the branch from
Lemma~\ref{lem:branch}.  The elements $w_n\in A_n$ are linearly independent.
Extend them to a Hamel basis and define a seminorm $p$ by
$p(w_n)=e^{n^2}$ and weight $1$ on the remaining basis vectors.  It is
continuous for the finest locally convex topology.  Given a continuous
seminorm $q$, set
\[
  M_q:=\max\{1,q(b):b\in S\}.
\]
Using $M_q^{-1}b_n,\ldots,M_q^{-1}b_1$ in the defining supremum gives
\[
  \|\mu_n\|_{p,q}
  \geq M_q^{-n}p(w_n)
  =M_q^{-n}e^{n^2},
\]
and therefore
\[
  \|\mu_n\|_{p,q}^{1/n}\geq M_q^{-1}e^n\longrightarrow\infty.
\]
This is incompatible with \eqref{eq:GN-condition}.

If $A$ is finite-dimensional, choose a norm $\|\cdot\|$ and $C\geq1$ with
$\|xy\|\leq C\|x\|\|y\|$.  For every seminorm $p$, there is a constant
$c_p$ such that $p(x)\leq c_p\|x\|$.  Taking $q=\|\cdot\|$ gives
$\|\mu_n\|_{p,q}\leq c_pC^{n-1}$, so
\eqref{eq:GN-condition} holds whenever $rC<1$.
\end{proof}

\begin{proof}[Proof of Theorem~\ref{thm:classification}]
If $A$ is finite-dimensional, then $H_A$ and $G_A$ are finite-dimensional
Lie groups and hence $C^0$-regular
\cite[Theorem~6.3.4]{GlocknerNeebBook2026}.  In particular, both groups are
$C^0$-semiregular and regular in the sense of Milnor.
Proposition~\ref{prop:growth} gives the growth condition.  If $A$ is infinite-dimensional,
Corollary~\ref{cor:all-k} shows that neither $H_A$ nor $G_A$ is even
$C^\infty$-semiregular, and hence neither group is regular in the sense of
Milnor.  Proposition~\ref{prop:growth} excludes the growth condition.  This
proves all equivalences.
\end{proof}

\begin{proof}[Proof of Theorem~\ref{thm:existence}]
Choose an algebra $A$ from Corollary~\ref{cor:existence} and apply
Theorem~\ref{thm:graded-escape}.
\end{proof}

\begin{proof}[Proof of Corollary~\ref{cor:cia-consequence}]
Choose the same algebra $A$.  Theorem~\ref{thm:graded-escape} gives the
non-semiregular full unit group, and Proposition~\ref{prop:growth} gives the
failure of the multiplication-growth condition.
\end{proof}

\section{Suspension and the failure of exponentiation}
\label{sec:suspension}

Let $H:=H_A$ be the group from Theorem~\ref{thm:graded-escape}, write
$\mathfrak h:=\mathbf L(H)=A$, and fix the control $\gamma_1$ of
Section~\ref{sec:control} with $\lambda=1$.  Thus
\begin{equation}
  V'(t)=\gamma_1(t)V(t),
  \qquad V(0)=e_H,
  \label{eq:suspension-right-equation}
\end{equation}
has no $H$-valued $C^1$ solution on $[0,1]$.  Put
$\gamma:=-\gamma_1$.

\begin{lemma}[Periodic realization of the bad control]
\label{lem:periodic-bad-control}
The left evolution equation
\begin{equation}
  u'(t)=T_{e_H}L_{u(t)}\gamma(t),
  \qquad u(0)=e_H,
  \label{eq:left-bad-equation}
\end{equation}
has no $H$-valued $C^1$ solution.  Moreover, $\gamma$ descends to a smooth
map $\widetilde\gamma:S^1=\mathbb R/\mathbb Z\to\mathfrak h$ satisfying
$\widetilde\gamma([t])=\gamma(t)$ for $0\leq t\leq1$.
\end{lemma}

\begin{proof}
If $u$ solved \eqref{eq:left-bad-equation}, then, in the unit group of
$\mathcal B_A$, the curve $V:=u^{-1}$ would satisfy
\[
  V'=-Vu'V=-V(u\gamma)V=\gamma_1V,
  \qquad V(0)=e_H,
\]
contradicting \eqref{eq:suspension-right-equation}.  By construction,
$\gamma_1$ vanishes near $0$, and Lemma~\ref{lem:smooth-pulses} shows that
all its derivatives vanish at $1$.  Hence $\gamma$ has matching zero jets at
the endpoints, so its $1$-periodic extension to $\mathbb R$ is smooth and
factors through $S^1$.
\end{proof}

Set
\[
  G:=C^\infty(S^1,H),
  \qquad
  \mathfrak g:=C^\infty(S^1,\mathfrak h),
\]
with pointwise multiplication, and let $\mathbb R$ act by translations,
\begin{equation}
  \alpha_z(F)([s]):=F([s+z]).
  \label{eq:translation-action}
\end{equation}

\begin{proposition}[The suspension group]
\label{prop:suspension-group}
The mapping group $G$ is a Lie group modelled on the complete Hausdorff
locally convex space $\mathfrak g$, with $\mathbf L(G)=\mathfrak g$.  The
maps
\[
  \operatorname{ev}_s:G\to H,
  \qquad F\mapsto F(s),
\]
are smooth homomorphisms, and \eqref{eq:translation-action} is a smooth action
by Lie group automorphisms.  Consequently,
\[
  \widetilde G:=G\rtimes_\alpha\mathbb R,
  \qquad
  (F,z)(K,w)=\bigl(F\alpha_z(K),z+w\bigr),
\]
is a Lie group modelled on the complete Hausdorff space
$\mathfrak g\times\mathbb R$.  Under the vector-space identification
$\mathbf L(\widetilde G)=\mathfrak g\times\mathbb R$, one has
\begin{align}
  \bigl(\mathbf L(\alpha_z)X\bigr)([s])&=X([s+z]),
  \label{eq:derived-translation-action}\\
  T_{e_{\widetilde G}}L_{(F,z)}(X,a)
  &=\bigl(T_{e_G}L_F(\mathbf L(\alpha_z)X),a\bigr).
  \label{eq:left-translation-semidirect}
\end{align}
\end{proposition}

\begin{proof}
The assertions about $G$, its Lie algebra, and completeness follow from the
standard compact-source mapping-group theorem
\cite[Theorems~I.2(1) and I.3(1)]{NeebWagemann2007}.  Evaluation is smooth
and has differential $X\mapsto X(s)$.  Joint smoothness of $\alpha$ follows
from the smooth exponential law applied to
$(z,F,s)\mapsto F(s+z)$.  The semidirect-product construction then gives
$\widetilde G$.  Formula \eqref{eq:derived-translation-action} follows by
differentiating precomposition, and differentiating
$(K,w)\mapsto(F\alpha_z(K),z+w)$ at $(e_G,0)$ gives
\eqref{eq:left-translation-semidirect}.
\end{proof}

Set
\begin{equation}
  g:=(\widetilde\gamma,1)
  \in\mathbf L(\widetilde G).
  \label{eq:suspended-element}
\end{equation}

\begin{theorem}[A non-integrable Lie algebra element]
\label{thm:2.3}
There is no $C^1$ curve $U:[0,1]\to\widetilde G$ satisfying
\begin{equation}
  U(0)=e_{\widetilde G},
  \qquad
  U'(t)=T_{e_{\widetilde G}}L_{U(t)}g.
  \label{eq:suspended-autonomous-equation}
\end{equation}
In particular, $g$ generates no one-parameter subgroup of $\widetilde G$.
\end{theorem}

\begin{proof}
Suppose that $U(t)=(F_t,z_t)$ solves
\eqref{eq:suspended-autonomous-equation}.  By
\eqref{eq:left-translation-semidirect},
\[
  z_t'=1,
  \qquad
  F_t'=T_{e_G}L_{F_t}
  \bigl(\mathbf L(\alpha_{z_t})\widetilde\gamma\bigr),
  \qquad
  (F_0,z_0)=(e_G,0).
\]
Hence $z_t=t$.  Define $u(t):=F_t([0])$.  Since evaluation is a smooth
homomorphism, it intertwines left translations; therefore
\begin{align*}
  u'(t)
  &=T_{e_H}L_{u(t)}
    \left(\bigl(\mathbf L(\alpha_t)\widetilde\gamma\bigr)([0])\right)\\
  &=T_{e_H}L_{u(t)}\widetilde\gamma([t])
   =T_{e_H}L_{u(t)}\gamma(t),
  \qquad u(0)=e_H.
\end{align*}
This contradicts Lemma~\ref{lem:periodic-bad-control}.

If $g$ generated a smooth one-parameter subgroup $\sigma$, then
$\sigma(t+s)=\sigma(t)\sigma(s)$ would imply
\[
  \sigma'(t)=T_{e_{\widetilde G}}L_{\sigma(t)}\sigma'(0)
  =T_{e_{\widetilde G}}L_{\sigma(t)}g.
\]
Its restriction to $[0,1]$ would therefore solve
\eqref{eq:suspended-autonomous-equation}, a contradiction.
\end{proof}

\begin{proof}[Proof of Theorem~\ref{thm:nonexponentiable}]
Proposition~\ref{prop:suspension-group} shows that $\widetilde G$ has a
complete Hausdorff locally convex model.  By Theorem~\ref{thm:2.3}, the
element $g$ generates no one-parameter subgroup.  Hence no exponential
function can be defined on all of $\mathbf L(\widetilde G)$, since such a
function would make $t\mapsto\exp_{\widetilde G}(tg)$ the one-parameter
subgroup generated by $g$.
\end{proof}

\medskip
\noindent\textbf{Acknowledgments.}
The authors thank the DeepMath team for its agent support.  The initial ideas
and inspiration for this work came from the authors and were further
developed through discussions with DeepMath agents.  GPT models were
then used to carry out the constructions and generate the manuscript
text.  The resulting arguments were subjected to adversarial review in
multiple separate conversations.  
Finally, the authors have manually checked the manuscript and concluded it sound.

The authors present special thanks to Professor Karl-Hermann Neeb 
for his detailed comments on an earlier version.


\begin{thebibliography}{99}

\bibitem{GlocknerCIA2002}
H.~Gl\"ockner,
\emph{Algebras whose groups of units are Lie groups},
Studia Math. \textbf{153} (2002), no.~2, 147--177,
\href{https://doi.org/10.4064/sm153-2-4}{doi:10.4064/sm153-2-4}.

\bibitem{GlocknerSemiregularity}
H.~Gl\"ockner,
\emph{Regularity properties of infinite-dimensional Lie groups, and
semiregularity},
2012, \href{https://doi.org/10.48550/arXiv.1208.0715}{arXiv:1208.0715}.

\bibitem{GlocknerDirectLimits2021}
H.~Gl\"ockner,
\emph{Direct limits of regular Lie groups},
Math. Nachr. \textbf{294} (2021), no.~1, 74--81,
\href{https://doi.org/10.1002/mana.201900073}{doi:10.1002/mana.201900073}.

\bibitem{GlocknerNeebCIA}
H.~Gl\"ockner and K.-H.~Neeb,
\emph{When unit groups of continuous inverse algebras are regular Lie groups},
Studia Math. \textbf{211} (2012), no.~2, 95--109,
\href{https://doi.org/10.4064/sm211-2-1}{doi:10.4064/sm211-2-1}.

\bibitem{GlocknerNeebBook2026}
H.~Gl\"ockner and K.-H.~Neeb,
\emph{Infinite-Dimensional Lie Groups},
preliminary monograph, 2026,
\href{https://doi.org/10.48550/arXiv.2602.12362}{arXiv:2602.12362}.

\bibitem{Golod1964}
E.~S.~Golod,
\emph{On nil-algebras and finitely approximable $p$-groups},
Izv. Akad. Nauk SSSR Ser. Mat. \textbf{28} (1964), no.~2, 273--276;
English transl., Amer. Math. Soc. Transl. (2) \textbf{48} (1965), 103--106.

\bibitem{Milnor1984}
J.~Milnor,
\emph{Remarks on infinite-dimensional Lie groups},
in: Relativity, Groups and Topology II (Les Houches, 1983),
North-Holland, Amsterdam, 1984, pp.~1007--1057.

\bibitem{Neeb2006}
K.-H.~Neeb,
\emph{Towards a Lie theory of locally convex groups},
Japan. J. Math. \textbf{1} (2006), no.~2, 291--468,
\href{https://doi.org/10.1007/s11537-006-0606-y}{doi:10.1007/s11537-006-0606-y}.

\bibitem{NeebRusso2024}
K.-H.~Neeb and F.~G.~Russo,
\emph{Ground state representations of topological groups},
Math. Ann. \textbf{388} (2024), 615--674,
\href{https://doi.org/10.1007/s00208-022-02531-4}{doi:10.1007/s00208-022-02531-4}.

\bibitem{RegevRegev2009}
A.~Regev and A.~Regev,
\emph{The Golod--Shafarevich counter-example without Hilbert series},
in: Groups, Rings and Group Rings, Contemp. Math. \textbf{499},
Amer. Math. Soc., Providence, RI, 2009,
\href{https://doi.org/10.1090/conm/499/09808}{doi:10.1090/conm/499/09808}.

\bibitem{SchaeferWolff}
H.~H.~Schaefer and M.~P.~Wolff,
\emph{Topological Vector Spaces}, 2nd ed.,
Graduate Texts in Mathematics, vol.~3, Springer, New York, 1999.

\bibitem{NeebWagemann2007}
K.-H.~Neeb and F.~Wagemann,
\emph{Lie group structures on groups of smooth and holomorphic maps on
non-compact manifolds},
Geom. Dedicata \textbf{134} (2008), 17--60,
\href{https://doi.org/10.1007/s10711-008-9244-2}{doi:10.1007/s10711-008-9244-2}.

\end{thebibliography}
\end{document}